\documentclass[12pt]{amsart}

\usepackage{amsmath}
\usepackage{amssymb}
\usepackage{fancybox}
\usepackage{eucal}
\usepackage{amsthm}
\usepackage{amscd}
\usepackage{hyperref}
\usepackage{wasysym}
\usepackage{url}

\DeclareMathOperator{\Ann}{Ann}
\DeclareMathOperator{\Pow}{Pow}
\DeclareMathOperator{\Cell}{Cell}

\begin{document}

\newcommand{\card}[1]{\ensuremath{\left |  #1 \right |}}
\newcommand{\ep}{\ensuremath{\;\Box}}
\newcommand{\N}[2]{\ensuremath{N^{#1}_{#2}}}
\newcommand{\al}{\ensuremath{\alpha}}
\newcommand{\be}{\ensuremath{\beta}}
\newcommand{\ga}{\ensuremath{\gamma}}
\newcommand{\de}{\ensuremath{\delta}}
\newcommand{\la}{\ensuremath{\lambda}}
\newcommand{\ob}[1]{\ensuremath{\overline{#1}}}
\newcommand{\acl}{\ensuremath{\mbox{acl}}}
\newcommand{\RR}{\ensuremath{\mathbb{R}}}
\newcommand{\NN}{\ensuremath{\mathbb{N}}}
\newcommand{\ZZ}{\ensuremath{\mathbb{Z}}}
\newcommand{\CC}{\ensuremath{\mathbb{C}}}
\newcommand{\dcl}{\ensuremath{\mbox{dcl}}}
\newcommand{\mf}[1]{\ensuremath{\mathfrak{#1}}}
\newcommand{\mc}[1]{\ensuremath{\mathcal{#1}}}
\newcommand{\QQ}{\ensuremath{\mathbb{Q}}}
\newcommand{\ssk}{\smallskip}
\renewcommand{\b}{\mathfrak{b}}
\newcommand{\f}{\mathfrak{f}}
\newcommand{\ralg}{\widetilde{\mathbb{Q}}}
\newcommand{\lam}{\lambda}
\newcommand{\lek}{\leq_{K_{ap}}}
\newcommand{\kap}{K_{ap}}
\newcommand{\res}{\restriction}
\newcommand{\lta}{\restriction \lambda \times \alpha}
\newcommand{\range}{\text{range}}
\newcommand{\sptp}{\text{sptp}}
\newcommand{\spcl}{\text{spcl}}
\newcommand{\FF}{\mathbb F}
\newcommand{\andd}{\wedge}
\newcommand{\mm}{\mathfrak m}


\newcommand{\Span}{\textup{Span}}
\newcommand{\tp}{\textup{tp}}

\def\Ind#1#2{#1\setbox0=\hbox{$#1x$}\kern\wd0\hbox to 0pt{\hss$#1\mid$\hss}
\lower.9\ht0\hbox to 0pt{\hss$#1\smile$\hss}\kern\wd0}
\def\ind{\mathop{\mathpalette\Ind{}}}
\def\Notind#1#2{#1\setbox0=\hbox{$#1x$}\kern\wd0\hbox to 0pt{\mathchardef
\nn=12854\hss$#1\nn$\kern1.4\wd0\hss}\hbox to
0pt{\hss$#1\mid$\hss}\lower.9\ht0 \hbox to
0pt{\hss$#1\smile$\hss}\kern\wd0}
\def\nind{\mathop{\mathpalette\Notind{}}}
\def\thind{\mathop{\mathpalette\Ind{}}^{\text{\th}} }
\def\nthind{\mathop{\mathpalette\Notind{}}^{\text{\th}} }
\def\uth{\text{U}^{\text{\th}} }
\def\oneind{\mathop{\mathpalette\Ind{}}^{1}}
\def\twoind{\mathop{\mathpalette\Ind{}}^{2}}
\def\noneind{\mathop{\mathpalette\Notind{}}^{1}}
\def\ntwoind{\mathop{\mathpalette\Notind{}}^{2}}


\newtheorem{assumps}[subsection]{Assumptions}
\newtheorem{defi}[subsection]{Definition}
\newtheorem{fact}[subsection]{Fact}
\newtheorem{example}[subsection]{Example}
\newtheorem{prop}[subsection]{Proposition}
\newtheorem{cor}[subsection]{Corollary}
\newtheorem{lem}[subsection]{Lemma}
\newtheorem{lemma}[subsection]{Lemma}
\newtheorem*{clm}{Claim}
\newtheorem{claim}[subsection]{Claim}
\newtheorem{rem}[subsection]{Remark}
\newtheorem{thm}[subsection]{Theorem}
\newtheorem{nota}[subsection]{Notation}
\newtheorem*{ques}{Question}
\newtheorem{goal}{Goal}
\newtheorem{conj}[subsection]{Conjecture}
\newtheorem*{term}{Terminology}
\newtheorem{subclaim}[subsection]{Subclaim}
\newtheorem*{facta}{Fact}
\newtheorem*{rema}{Remark}
\newtheorem{assump}{Assumption}
\newtheorem{definition}{Definition}
\newtheorem{remark}{Remark}
\newtheorem{red}{Reduction}

\title{Dp-Minimality: Basic facts and examples}

\author{Alfred Dolich, John Goodrick, and David Lippel}

\address{Department of Mathematics East Stroudsburg University, East Stroudsburg PA 18301}
\email{adolich@gmail.com}
\address{Department of Mathematics University of Maryland, College Park MD 20742}
\email{goodrick@math.umd.edu}
\address{Department of Mathematics Haverford College, Haverford PA 19041}
\email{dlippel@haverford.edu}

\date{\today}

\thanks{A portion of this paper was completed during the program ``Stable Methods in Unstable Context'' at
the Banff International Research Station.  The authors thank BIRS for their hospitality.}

\thanks{The first author would like to thank the Fields Institute for their hospitality during
the thematic program ``O-minimal Structures and Analytic Geometry'' during which a portion of this
paper was completed.}

\begin{abstract}  We study the notion of dp-minimality, beginning by providing
several essential facts about dp-minimality, establishing several equivalent definitions for dp-minimality, and
comparing dp-minimality to other minimality notions.
The majority of the rest of the paper is dedicated to examples.  We establish via
a simple proof that any weakly o-minimal theory is dp-minimal and then give an example
of a weakly o-minimal group  not obtained by adding traces of externally definable sets.  Next we give an
example of a  divisible
ordered Abelian group which is dp-minimal and not weakly o-minimal. Finally we
establish that the field of p-adic numbers is dp-minimal.
\end{abstract}

\maketitle

\section{introduction}

In this note we study many basic properties of dp-minimality as well as developing
several fundamental examples.  Dp-minimality---see Definition~\ref{basic}---was introduced by Shelah in~\cite{sh-stdep}
as possibly the strongest of a family of notions implying that a theory does not
have the independence property---for which see~\cite{sh1}.
  The study of dp-minimality beyond
Shelah's original work was continued by Onshuus and Usvyatsov in~\cite{alfdpmin} focusing
primarily on the stable case and by the second author in~\cite{gooddp} primarily in
the case of theories expanding the theory of divisible ordered Abelian groups.
Our goal in this paper is to provide many basic foundational facts on dp-minimality
as well as to explore concrete examples of dp-minimality in the ordered context as well
as the valued field context.  Much of our motivation arises out of the program of attempting
to explore the impact of abstract model theoretic notions, such as dp-minimality, in concrete
situations such as ordered model theory on the reals or the study of valued fields.  As such this
note may be seen as providing some ground work for further study in this direction---for example Pierre
Simon~\cite{simon} has recently shown that an infinite definable subset of a dp-minimal
divisible ordered Abelian group must have interior.

We give a brief outline of the content of our note.  Section~\ref{abcs} is dedicated to
 definitions and providing several relevant background facts on dp-minimality.  Many of these
facts are inherent in~\cite{sh-stdep} but we isolate them here and provide straightforward
proofs.  Section~\ref{vc} is a brief discussion of the relationship of
dp-minimality to some other minimality notions.  In Section~\ref{wom} we focus weak on o-minimality---for which
see~\cite{weakominfield}---and show that a weakly o-minimal theory is dp-minimal as well providing an example of
a weakly o-minimal group which is not obtained by expanding on o-minimal structure by
convex sets.  Work in Section 3 as well as results from~\cite{gooddp} indicate that a dp-minimal
theory expanding that of
divisible ordered Abelian groups has some similarity to a weakly o-minimal theory and we may
naturally ask whether any such theory is weakly o-minimal.  Section~\ref{nwom} provides a negative
answer via an example arising form the valued field context.  Our final section is dedicated
to showing that the theory of the p-adic field is dp-minimal.

\section{basic facts on dp-minimality}\label{abcs}

We develop several basics facts about dp-minimality.  The vast majority of the material found
below is inherent in Shelah's paper~\cite{sh-stdep}, but typically in the more general
context of strong dependence.  We provide proofs of these various facts for clarity and
ease of exposition.  Recall:

\begin{defi}\label{ict}{\em Fix a structure $\mf{M}$,  An {\em ICT pattern} in $\mf{M}$ consists of  a pair of
formulae $\phi(x, \ob{y})$ and $\psi(x,\ob{y})$; and
sequences  $\{\ob{a}_i : i \in \omega\}$ and $\{\ob{b}_i : i \in \omega\}$ from $M$ so
that for all $i,j \in \omega$ the following is consistent:
$$\phi(x,\ob{a}_i) \wedge \psi(x,\ob{b}_j) \wedge \bigwedge_{l \not =i}\neg \phi(x, \ob{a}_l)
\wedge \bigwedge_{k \not= j}\neg\psi(x,\ob{b}_k).$$}
\end{defi}

\begin{rema}{\em  Definition~\ref{ict} should more formally be referred to as an ICT pattern of {\em depth two}
but in this paper we only consider such ICT patterns and thus we omit this extra terminology.}
\end{rema}

\begin{defi}\label{basic}{\em  A theory $T$ is said to be {\em dp-minimal} if in no model
$M \models T$ is there an ICT pattern.
}
\end{defi}

It is often very convenient to use the following definition and fact.

\begin{defi}{\em We say two sequences $\{\ob{a}_i : i \in I\}$ and $\{\ob{b}_j : j \in J\}$ are {\em mutually
indiscernible} if $\{\ob{a}_i : i \in I\}$ is indiscernible over $\bigcup_{j \in J}\ob{b}_j$ and
$\{\ob{b}_j : j \in J\}$ is indiscernible over $\bigcup_{i \in I}\ob{a}_i$.  We call an ICT pattern mutually
indiscernible if the witnessing sequences are mutually indiscernible.}
\end{defi}

\begin{fact}\label{indisc} $T$ is dp-minimal if and only if in no model $\mf{M} \models T$ is there a
mutually indiscernible ICT pattern.
\end{fact}

\begin{proof}  This is a simple application of compactness and Ramsey's theorem. \end{proof}

Before continuing we should mention another alternative characterization of dp-minimality.  To this end we have
the following definition.

\begin{defi}{\em  A theory $T$ is said to be {\em inp-minimal} if there is {\bf no} model of $\mf{M}$ of $T$ formulae
$\phi(x, \ob{y})$ and $\psi(x, \ob{y})$, natural numbers $k_0$ and $k_1$, and sequences $\{\ob{a}_i : i \in \omega\}$ and
$\{\ob{b}_i : i \in \omega\}$ so that $\{\phi(x, \ob{a}_i) : i \in \omega\}$ is $k_0$-inconsistent, $\{\psi(x, \ob{b}_j) :
 j \in \omega\}$ is $k_1$-inconsistent and for any $i,j \in \omega$ the formula $\phi(x, \ob{a}_i) \wedge \psi(x, \ob{b}_j)$
is consistent.
}
\end{defi}

With this definition we have the following fact:

\begin{fact}{\cite[Lemma 2.11]{alfdpmin}} If $T$ does not have the independence property and is inp-minimal then $T$ is dp-minimal.
\end{fact}

The next fact is extremely useful in showing that a theory is not dp-minimal.  It is key
in establishing the relationship between dp-minimality and indiscernible sequences that follows.

\begin{fact}\label{config}  Let $T$ be a complete theory, we work in a monster model $\mf{C}$.
Suppose there are formulae $\phi_0(x, \ob{y})$ and $\phi_1(x, \ob{y})$ and
mutually indiscernible sequences $\{\ob{a}_i : i \in \omega\}$ and $\{\ob{b}_i : i \in \omega\}$ so
that \[\phi_0(x, \ob{a}_0) \wedge \neg \phi_0(x, \ob{a}_1) \wedge \phi_1(x, \ob{b}_0) \wedge
\neg \phi_1(x, \ob{b}_1) \] is consistent then $T$ is not dp-minimal.
\end{fact}

\begin{proof} By compactness there are mutually indiscernible sequences $\ob{a}_i$ for $i \in \ZZ$ and
$\ob{b}_i$ for $i \in \ZZ$ and $c$ so that
\[ \models \phi_0(c, \ob{a}_0) \wedge \neg \phi_0(c, \ob{a}_1) \wedge \phi_1(c, \ob{b}_0), \wedge \neg
\phi_1(c, \ob{b}_1).\]
By applying compactness and Ramsey's theorem we assume that $\{\ob{a}_i: i<0\}$ and
$\{\ob{a}_i: i>1\}$ are both indiscernible over $\bigcup_{i \in \omega}\ob{b}_i \cup \{c\}$ as well as
that $\{\ob{b}_i: i<0\}$ and $\{\ob{b}_i: i>1\}$ are both indiscernible over
$\bigcup_{i \in \omega}\ob{a}_i \cup \{c\}$.  Let $\ob{d}_i=\ob{a}_{2i}^{\frown}\ob{a}_{2i+1}$
and $\ob{e}_i=\ob{b}_{2i}^{\frown}\ob{b}_{2i+1}$ for $i \in \ZZ$.  Note that these two sequences
are mutually indiscernible.  Let $\psi_0(x, \ob{y}_0\ob{y}_1)$
be $\phi_0(x, \ob{y}_0) \leftrightarrow \neg \phi_0(x, \ob{y}_1)$
and $\psi_1(x, \ob{y}_0\ob{y}_1)$ be $\phi_1(x,\ob{y}_0) \leftrightarrow \neg \phi_1(x, \ob{y}_1)$.
Then  $\mf{C}\models \psi_0(c, \ob{d}_0) \wedge \psi_1(c, \ob{e}_0)$ and if $i \not=0$ then
$\models \neg \psi_0(c, \ob{d}_i) \wedge \neg \psi_1(c, \ob{e}_i)$ by the indiscernibility
assumptions.  It follows
that $\psi_0$, $\psi_1$, $\{\ob{d}_i : i \in \omega\}$ and $\{\ob{e}_i : i \in \omega\}$ witness that
$T$ is not dp-minimal. \end{proof}

Using the identical proof as above we show:

\begin{fact}\label{gen}  The following are equivalent:
\begin{enumerate}

\item There are formulae $\phi_i(\ob{x},\ob{y})$ for $1 \leq i \leq N$ and
sequences $\ob{a}_j^i$ for $1 \leq i \leq N$ and $ j \in \omega$ so that
for every $\eta: \{1, \dots, N\} \to \omega$ the type:
\[\bigwedge_{1 \leq i \leq N}\phi_i(\ob{x}, \ob{a}_{\eta(i)}^i) \wedge
\bigwedge_{1 \leq i \leq N}\bigwedge_{j \not= \eta(i)}\neg \phi_i(\ob{x}, \ob{a}_j^i)\]
is consistent.
\item There are formulae $\psi_i(\ob{x},\ob{y})$ for $1 \leq i \leq N$ and mutually
indiscernible sequences $\{a_j^i: j \in \omega\}$ for $1 \leq i \leq N$ so that
the type:
\[ \bigwedge_{1 \leq i \leq N}\psi_i(\ob{x},\ob{a}^i_0) \wedge \bigwedge_{1 \leq i \leq N}\neg
\psi_i(\ob{x},\ob{a}^i_1)\] is consistent
\end{enumerate}
\end{fact}

As in the case of the independence property and strong dependence we have a characterization of
dp-minimality in terms of splitting indiscernible sequences.

\begin{fact}\label{ind} The following are equivalent for a theory $T$:

\begin{enumerate}
\item $T$ is dp-minimal.
\item If $\{\ob{a}_i : i \in I\}$ is an indiscernible sequence and $c$ is an element then
there is a partition of $I$ into finitely many convex sets $I_0, \dots, I_n$, at most two of which are infinite,
so that for any $1 \leq l \leq n$ if $i,j \in I_l$ then $tp(\ob{a}_i/c)=tp(\ob{a}_j/c)$.
\item If $\{\ob{a}_i : i \in I\}$ is an indiscernible sequence and $c$ is an element then
there is a partition of $I$ into finitely many convex sets $I_0, \dots, I_n$, at most two of which are infinite,
so that for any $1 \leq l \leq n$ the sequence $\{\ob{a}_i : i \in I_l\}$ is indiscernible over $c$.
\end{enumerate}
\end{fact}

\begin{proof}$(1) \Rightarrow  (3)$:
Suppose $T$ is dp-minimal and for contradiction suppose that there is an indiscernible sequence (which for notational simplicity we
assume consists of singletons) $\{a_i : i \in I\}$ and an element $c$ witnessing the failure of $(3)$.  Without loss
of generality we assume that $I$ is a sufficiently saturated dense linear order without endpoints.  By~\cite[Corollary 6]{adler-nip}
there is an initial segment $I_0 \subseteq I$ and a final segment $I_1 \subseteq I$ so that the sequences $\{a_i: i \in I_0\}$ and
$\{a_i: i \in I_1 \}$ are indiscernible over $c$.  We may choose $I_0$ to be maximal in the sense that for no convex $J$ with
$I_0 \subset J \subseteq I$ is $\{a_j: j \in J \}$ indiscernible over $c$ and similarly for $I_1$.  If $I \setminus (I_0 \cup I_1)$
is finite then $(3)$ holds so $I \setminus (I_0 \cup I_1)$ is infinite and thus contains an interval.  Let $J_0$ and $J_1$ be disjoint
convex sets so that $I_0 \subset J_0 \subset I$ and $I_1 \subset J_1 \subset I$.
We find $j^*_1 < \dots < j^*_n \in J_0$ and a formula $\phi(x, y_1, \dots, y_n)$ so
that $\neg \phi(c, a_{i_1}, \dots a_{i_n})$ holds for all $i_1< \dots < i_n \in I_0$ but $\phi(c, a_{j^*_1}, \dots, a_{j^*_m})$ holds.
We choose a sequence of $n$-tuples $\ob{d}_i$ for $i \in \omega$ as follows: Let $\ob{d}_0=a_{j^*_1}, \dots, a_{j^*_n}$, if $i>0$
let $\ob{d}_i=a_{i_1}, \dots, a_{i_n}$ where $i_1, \dots, i_n$ is any increasing sequence of elements  with $i_k \in I_0$
for all $1 \leq k \leq n$ and so that
$i_n < \min\{l \in I : a_l \in \ob{d}_{i-1}\}$.  Note that $\{\ob{d}_i: i \in \omega\}$ is indiscernible.  Similarly we may find
$k^*_{j_1} < \dots < k^*_{j_m} \in J_1$ and a formula $\psi(x, y_1, \dots, y_m)$ so that
$\neg\psi(c, a_{i_1}, \dots, a_{i_m})$ holds for all $i_1 < \dots < i_m \in I_1$ but $\psi(c, a_{k^*_1}, \dots, a_{k^*_m})$ holds.
We build a sequence of $m$-tuples $\{\ob{e}_i : i \in \omega\}$ as follows: $\ob{e}_0 = a_{k^*_1}, \dots, a_{k^*_m}$.  For $i>0$ let
$\ob{e}_i={a}_{i_1}, \dots, a_{i_m}$ where $i_1, \dots, i_m$ is any increasing sequence from $I_1$ and $i_1>\max\{l \in I : a_l \in \ob{e}_{i-1}\}$.
Notice that $\{\ob{d}_i : i \in \omega\}$ and $\{\ob{e}_j: j \in \omega\}$ are mutually indiscernible sequences.
 Thus the formulae $\phi(x, \ob{y}_1, \dots, \ob{y}_n)$
 and $\psi(x, \ob{y}_1, \dots, \ob{y}_m)$ and the sequences $\ob{d}_i$, $\ob{e}_i$ for $i \in \omega$ form
a configuration as found in Fact~\ref{config} and $T$ is not dp-minimal.

\smallskip

$(3) \Rightarrow (2)$: Immediate.

\smallskip

$(2) \Rightarrow (1)$.  For contradiction suppose there are formulae $\phi(x, \ob{y})$ and $\psi(x, \ob{y})$ and
mutually indiscernible sequences $\{\ob{a}_i : i \in 3 \times \omega\}$ and  $\{\ob{b}_i : i \in 3  \times \omega\}$
which witness the failure of
dp-minimality.  Such sequences must exist by compactness.
For $i \in 3 \times \omega$ let $\ob{c}_i=\ob{a}_i^\frown \ob{b}_i$. Note this is
an indiscernible sequence.  Pick $d$ realizing the type:
\[\phi(x, \ob{a}_{\omega}) \wedge \psi(x, \ob{a}_{2 \times \omega}) \wedge \bigwedge_{i \not= \omega}
\neg \phi(x, \ob{a}_i) \wedge \bigwedge_{j \not= 2 \times \omega}\neg \psi(x, \ob{b}_j).\]
If $i \not= \omega$ or $2 \times \omega$ then $\models \neg \phi(d, \ob{a}_i) \wedge \neg \psi(d, \ob{b}_i)$,
if $i = \omega$ then $\models \phi(d, \ob{a}_{\omega}) \wedge \neg \psi(d,\ob{b}_{\omega})$ and
if $i = 2 \times \omega$ then $\models \neg \phi(d, \ob{a}_{2 \times \omega}) \wedge \psi(d, \ob{b}_{2 \times \omega})$.
It follows that the indiscernible sequence $\{\ob{c}_i : i \in 3 \times \omega\}$ can not be decomposed
 into convex subsets of which at most two are infinite so that the type over $d$ is constant on each convex
set.  Hence $(2)$ fails. \end{proof}

We use the preceding fact to provide a characterization of dp-minimality which allows us to
consider formulae with more than one free variable in a variant of Definition~\ref{basic} as well
as to consider sets of parameters of arbitrary size in a variant of Fact~\ref{ind}.  The proof is
immediate using Facts~\ref{gen} and~\ref{ind}.

\begin{fact}\label{many}  For any theory $T$ the following are equivalent.

\begin{enumerate}
\item $T$ is dp-minimal.
\item If $\{\ob{a}_i : i \in I\}$ is an indiscernible sequence and $\ob{c}$ is an $n$-tuple from the universe
then there is a partition of $I$ into convex sets $I_1, \dots, I_l$ of which at most $2^n$ are infinite so that
if $1 \leq l \leq n$ and $i, j \in I_l$ then $tp(\ob{a}_i/\ob{c})=tp(\ob{a}_j/\ob{c})$.
\item If $\{\ob{a}_i : i \in I\}$ is an indiscernible sequence and $\ob{c}$ is an $n$-tuple from the universe
then there is a partition of $I$ into convex sets $I_1, \dots, I_l$ of which at most $2^n$ are infinite so
that if $1 \leq l \leq n$ the sequence $\{\ob{a}_i : i \in I_l\}$ is indiscernible.
\item There is no sequence of  formulae $\phi_1(\ob{x}, \ob{y}), \dots, \phi_{2^n}(\ob{x}, \ob{y})$ with $\card{\ob{x}}=n$
and sequences $\{\ob{a}^j_i : i \in \omega\}$ with $1 \leq j \leq 2^n$  so that form any $\eta: \{1, \dots, 2^n\} \to \omega$
the type
\[ \bigwedge_{1 \leq k \leq 2^n} \phi_k(x, \ob{a}^k_{\eta(k)}) \wedge \bigwedge_{1 \leq k \leq 2^n}
\bigwedge_{\{t \in \omega: \eta(k) \not= t\}}\neg\phi_k(x, \ob{a}^k_t)\] is consistent.
\end{enumerate}
\end{fact}

Given a dp-minimal theory $T$ it is reasonable to ask if the bound in Fact~\ref{many}(4) may be improved from  $2^n$
 to $n+1$.  This holds in the stable case and we sketch the proof.

\begin{fact}If $T$ is dp-minimal and stable then there is no sequence of formulae
\[\phi_1(\ob{x}, \ob{y}), \dots, \phi_{n+1}(\ob{x}, \ob{y})\]with $\card{\ob{x}}=n$
and sequences $\{\ob{a}^j_i : i \in \omega\}$ with $1 \leq j \leq n$  so that for any $\eta: \{1, \dots, n+1\} \to \omega$
the type
\[ \bigwedge_{1 \leq k \leq n+1} \phi_k(x, \ob{a}^k_{\eta(k)}) \wedge \bigwedge_{1 \leq k \leq n}
\bigwedge_{\{t \in \omega: \eta(k) \not= t\}}\neg\phi_k(x, \ob{a}^k_t)\] is consistent.
\end{fact}

\begin{proof}(Sketch)  By~\cite[Theorem 3.5]{alfdpmin}  $T$ is stable and dp-minimal if and only if every 1-type has weight $1$.
Thus if $T$ is stable and dp-minimal every $n$-type has weight at most $n$.  Apply~\cite[Lemma 2.3 and Lemma 2.11]{alfdpmin}
and the result follows.\end{proof}

This fact also holds for weakly o-minimal $T$ (we defer discussion of this to section~\ref{vc}).  In fact it is tempting to
restate the result with stable replaced by rosy and use \thorn-weight as in~\cite{alforthog} but at the
moment it is not clear if the result follows.  Overall we do not know whether these improved bounds hold for a general
 dp-minimal theory.

We finish with two facts which are useful in studying specific examples.  The first of these is particularly useful
when studying theories which admit some type of cell decomposition---viz. the p-adics.

\begin{fact}\label{union}  Suppose that $\phi(x, \ob{y})$ and $\psi(x, \ob{y})$ are formulae witnessing that
a theory $T$ is not dp-minimal.  Further suppose that $\phi(x, \ob{y})$ is
$\phi_1(x, \ob{y}) \vee \dots \vee \phi_n(x, \ob{y})$.  Then for some $1 \leq l \leq n$ $\phi_l(x, \ob{y})$ and
$\psi(x, \ob{y})$ witness that $T$ is not dp-minimal.
\end{fact}

\begin{proof}  There are  mutually indiscernible sequences $\{\ob{a}_i : i \in \ZZ\}$ and $\ob{b}_i : i \in \ZZ\}$ so
that for any $i^*,j^* \in \ZZ$  the type
\[\phi(x, \ob{a}_{i^*}) \wedge \psi(x, \ob{b}_{j^*}) \wedge \bigwedge_{i \not= i^*}\neg\phi(x, \ob{a}_i) \wedge
\bigwedge_{j \not= j^*}\neg \psi(x, \ob{b}_j)\] is consistent.  For each $i^* \in \ZZ$ there is $l(i^*) \in \{1, \dots, n\}$
so that
\[\phi_{l(i^*)}(x, \ob{a}_{i^*})  \wedge \psi(x, \ob{b}_{0}) \wedge \bigwedge_{i \not=i^* }\neg\phi(x, \ob{a}_i) \wedge
\bigwedge_{j \not= j^*}\neg \psi(x, \ob{b}_j)\] is consistent.  Thus for some infinite $I \subseteq \ZZ$ and $1 \leq l^*
\leq n$ if $i \in I$ then $l(i)=l^*$.  It follows that the formulae $\phi_{l^*}(x, \ob{y})$ and
$\psi(x, \ob{y})$ together with the sequences $\{\ob{a}_i : i \in I\}$ and $\{\ob{b}_i : i \in \ZZ\}$ witness that
$T$ is not dp-minimal.\end{proof}

Our final fact shows that a counterexample to dp-minimality may always be found which uses only a single formula rather than
two as in Definition~\ref{basic}.

\begin{fact}\label{one}  Suppose that $T$ is not dp-minimal.  Then there is a formula $\theta(x, \ob{y})$ and sequences
$\{\ob{c}_i : i \in \omega\}$ and $\{\ob{d}_i : i \in \omega\}$ so that for any $i \not= j$ the type
\[\theta(x, \ob{c}_i) \wedge \theta(x, \ob{d}_j) \wedge \bigwedge_{k \not= i}\neg \theta(x, \ob{c}_k) \wedge
\bigwedge_{l \not= j}\neg \theta(x, \ob{d}_l)\] is consistent.
\end{fact}

\begin{proof}  Suppose that the formulae $\phi(x, \ob{y})$ and $\psi(x, \ob{y})$ and the
sequences $\{\ob{a}_i : i \in \omega + \omega\}$ and
$\{\ob{b}_j : j \in \omega + \omega\}$ witness that $T$ is not dp-minimal.
Let $\theta(x, \ob{y}_1,\ob{y}_2)$ be $\phi(x, \ob{y}_1) \vee \psi(x, \ob{y}_2)$.  For $i \in \omega$ let
 $\ob{c}_i$ be $\ob{a}_i\ob{b}_i$
 and for $j \in \omega$ let $\ob{d}_j$ be $\ob{a}_{\omega+j}\ob{b}_{\omega+j}$.  We claim this is as the fact requires.
Fix $i,j \in \omega$.  There is  $\alpha$ realizing the type
\[\phi(x, \ob{a}_j) \wedge \psi(x, \ob{b}_{\omega + j}) \wedge \bigwedge_{k \not= i}\neg \phi(x, \ob{a}_k) \wedge
\bigwedge_{l \not= \omega + j}\neg\psi(x, \ob{b}_l).\]
Thus  $\alpha$ realizes $\theta(x, \ob{c}_i)$ and $\theta(x, \ob{d}_j)$.  If $k \not= i$ then $\alpha$ realizes
$\neg \phi(x, \ob{a}_i)$ and $\neg \psi(x, \ob{b}_i)$ and hence realizes $\neg \theta(x, \ob{c}_i)$.  Finally if $l \not= j$
then $\alpha$ realizes $\neg \theta(x, \ob{a}_{\omega + l})$ and $\neg \theta(x, \ob{b}_{\omega+l})$ and hence $\neg \theta (x, \ob{d}_j)$
as desired.\end{proof}

\section{Relationship with similar notions}\label{vc}

In this brief section we examine the relationship between dp-minimality and various other strong forms of
dependence.

We begin with the notion of VC-density as studied in~\cite{vcdense}.
 For the ensuing definition we fix $\Delta(\ob{x},\ob{y})$ a finite
set of formulae where we consider $\ob{y}$ as the parameter variables.  If $A$ is
a set of $\card{\ob{y}}$-tuples we write $S^{\Delta}(A)$ for the
set of complete $\Delta$-types with parameters from $A$.

\begin{defi}{\em A theory $T$ has {\em VC-density one} if for any finite set of formulae $\Delta(\ob{x},\ob{y})$
there is a constant $C$ so that for any finite set, $A$, of $\card{\ob{y}}$-tuples $\card{S^{\Delta}(A)} \leq C\card{A}^{\card{\ob{x}}}$.}
\end{defi}

For example in \cite{vcdense} it is shown that
 any weakly o-minimal theory and any quasi o-minimal theory with definable bounds (for which see~\cite{qom}) has VC-density one.
  We have a
strong relationship between VC-density one and dp-minimality:

\begin{prop} If $T$ has VC-density one then $T$ is dp-minimal.
\end{prop}

\begin{proof}  Suppose that $T$ is not dp-minimal.  Apply Fact~\ref{one} to find a formula $\phi(x, \ob{y})$
and sequences $\ob{a}_i$ and $\ob{b}_j$ as described there.  For $N \in \NN$ let
\[A_N=\{\ob{a}_i : i \leq N\} \cup \{\ob{b}_j : j \leq N\}.\]
By the failure of dp-minimality we immediately see that $S^{\{\phi\}}(A_N) \geq \frac{1}{4}\card{A_N}^2$ for all
$N$.
Thus $T$ does not have VC-density one. \end{proof}

Note that the above proof only requires that we have that $\card{S^{\Delta}(A)} \leq C\card{A}$ for $\Delta$ consisting
of formulae of the form $\phi(x, \ob{y})$, i.e. with only one free variable.  We may use this fact to provide a novel proof for
 the dp-minimality
of the theory of algebraically closed valued fields.  This fact was already observed in \cite{adlervc} and as the proof is
somewhat technical we do not include it here.

\begin{cor} Any weakly o-minimal theory as well as any quasi o-minimal theory with definable bounds is
dp-minimal.
\end{cor}

We give a short direct proof of the dp-minimality of any weakly o-minimal theory in the next section.
Also under the assumption of VC-density one we obtain the improved bounds in Fact~\ref{many}(4) as discussed
in section~\ref{abcs}.  With the same proof as in the previous proposition we show:

\begin{prop}  If $T$ has VC-density one there is no sequence of formulae
\[\phi_1(\ob{x}, \ob{y}), \dots, \phi_{n+1}(\ob{x}, \ob{y})\]with $\card{\ob{x}}=n$
and sequences $\{\ob{a}^j_i : i \in \omega\}$ with $1 \leq j \leq n$  so that form any $\eta: \{1, \dots, n+1\} \to \omega$
the type
\[ \bigwedge_{1 \leq k \leq n+1} \phi_k(x, \ob{a}^k_{\eta(k)}) \wedge \bigwedge_{1 \leq k \leq n}
\bigwedge_{\{t \in \omega: \eta(k) \not= t\}}\neg\phi_k(x, \ob{a}^k_t)\] is consistent.
\end{prop}

We consider the notion of VC-minimality as introduced in~\cite{adlervc}.

\begin{defi}{\em  Fix a theory $T$ and a monster model $\mf{C} \models T$.
 $T$ is {\em VC-minimal} if there is a family of formulae $\Phi$ of the the
form $\phi(x, \ob{y})$ so that:
\begin{itemize}
\item  If $\phi(x,\ob{y}), \psi(x, \ob{y}) \in \Phi$ and $\ob{a}, \ob{b} \in C$ then
one of:
\begin{itemize}
\item $\phi(C,\ob{a}) \subseteq \psi(C, \ob{b})$,
\item $ \psi(C, \ob{b}) \subseteq \phi(C, \ob{a})$,
\item $\neg\phi(C, \ob{a})
\subseteq \psi(C, \ob{b})$,
\item  $\psi(C, \ob{b}) \subseteq \neg\phi(C, \ob{a}).$
\end{itemize}
\item If $X \subseteq C$ then there are a finite collection of $\phi_i(x, \ob{y})$ from $\Phi$ and
tuples $\ob{a}_i \in C$ so that $X$ is a Boolean combination of the sets $\phi_i(C, \ob{a}_i)$.
\end{itemize}}
\end{defi}

Once again we have a strong relation with dp-minimality.

\begin{prop}{\cite[Proposition 9]{adlervc}}  Any VC-minimal theory is dp-minimal.
\end{prop}

This implication may not be reversed as the following proposition shows.

\begin{prop} Let $\mc{L}$ be the language consisting of unary predicates $P_i$ with $i \in \omega_1$.
For any finite $I \subseteq \omega_1$ and $J \subseteq \omega_1$ with $I \cap J = \emptyset$ let
$\sigma_{I,J,n}$ be the sentence $\exists^{\geq n}(\bigwedge_{i \in I}P_ix \wedge \bigwedge_{j \in J}\neg P_jx)$.
Let $T = \{\sigma_{I,J,n}: \text{ for all } I, J, n\}$.   $T$ is complete, has quantifier elimination, is dp-minimal
(in fact has VC-density one),
but is not VC-minimal.
\end{prop}

\begin{proof} We sketch a proof and leave the details to the interested reader.  Completeness and quantifier elimination
for $T$ are elementary.   Quantifier elimination  and
Fact~\ref{ind}(2) yield that $T$ is dp-minimal.  For the non-VC minimality of $T$ use quantifier elimination
show that there is no family $\Phi$ of formulae $\phi(x, \ob{y})$ satisfying the compatibility conditions  in the
definition of VC-minimality
and so that any $P_i(x)$ can be obtained by a finite Boolean combination of sets defined by instances of $\Phi$. \end{proof}

Thus we have an example of a theory which has VC-density one (and hence is dp-minimal) but is not VC-minimal.
We do not have an example of a theory which is dp-minimal but does not have VC-density one.

\section{Weakly o-minimal theories}\label{wom}

In this section we focus on weakly o-minimal theories and structures.  To begin with we show that any
weakly o-minimal theory is dp-minimal.   Dp-minimality for weakly o-minimal
theories first appears in~\cite[Corollary 3.8]{alfdpmin} where it is shown that any weakly
o-minimal theory obtained via ``Shelah expansion''---which
we discuss below---is dp-minimal.  We show via an example that not all weakly o-minimal theories may be obtained this in
way.  The dp-minimality of weakly o-minimal theories is also established in~\cite{adlervc} and~\cite{vcdense}
by indirect arguments.  Here we give a short
direct proof.

\begin{thm}\label{weakomin}  A weakly o-minimal theory is dp-minimal.
\end{thm}

\begin{proof}  Suppose the result fails for a weakly o-minimal theory $T$ and fix $\mf{M} \models T$.  Let $\phi(x,\ob{y})$ and
$\psi(x, \ob{y})$ be formulae together with mutually indiscernible sequences $\{\ob{a}_i: i \in \omega\}$ and
$\{\ob{b}_i : i \in \omega\}$ forming an ICT pattern.  By Fact~\ref{union} we assume that
for all $i,j \in \omega$ that $\phi(x, \ob{a}_i)$ and $\psi(x, \ob{b}_j)$ define single convex sets.  To ease
notation let  $A_i=\phi(M, \ob{a}_i)$ and $B_i=\psi(x, \ob{b}_i)$.  We
assume these convex sets are infinite since otherwise we obtain a contradiction.
  For convex sets $X, Y \subseteq M$  write $X \prec Y$ if for any $y \in Y$ there
is $x \in X$ so that $x<y$, $X \not\subseteq Y$, and $Y \not\subseteq X$.  Notice that $\prec$ linearly orders
any family of convex sets satisfying no containment relations.  We assume that for all $i \in \omega$ that
$A_i$ and $B_i$ are bounded convex sets since otherwise we easily arrive at a contradiction.
Note that there can be no containments among any of the sets $A_i$ and $B_j$.  Without loss of generality
 assume that $A_0 \prec A_1$.
We show that $A_0 \cap A_1 \not= \emptyset$. If not then
 $A_0 \prec B_0 \prec A_1$.   If $B_0 \prec B_1$ then $B_1 \cap A_0 \subseteq A_0$ contradicting that these
sets form an ICT pattern.  But if $B_1 \prec B_0$ then $B_0 \prec A_0 \subseteq B_1$ also contradicting that the
sets form an ICT pattern.
  Thus $A_0 \cap A_1 \not= \emptyset$ and
symmetrically $B_0 \cap B_1 \not= \emptyset$.  To finish suppose that $B_0 \prec A_0$ (the case that $A_0 \prec B_0$ is symmetric).
Thus also $B_0 \prec A_1$, and by mutual indiscernibility $B_1 \prec A_0$.
If $B_1 \prec B_0$ then $B_1 \cap A_1 \subseteq B_0$ violating that these sets form an ICT pattern.
  But if $B_0 \prec B_1$ then $B_0 \cap A_1 \subseteq
B_1$ also violating that these sets form on ICT pattern.  Thus in all cases there is  a contradiction,\end{proof}

This theorem immediately yields:

\begin{cor} Any o-minimal theory is dp-minimal.
\end{cor}

Theorem~\ref{weakomin} requires that the {\it theory} $T$ be weakly o-minimal.  We ask if the
theory of a weakly o-minimal structure $\mf{M}$ is dp-minimal---recall that a weakly o-minimal structure need
not have weakly o-minimal theory.
We do not know of an example of a weakly o-minimal structure whose theory is not dp-minimal.  Given the
close relationship between a theory being weakly o-minimal and elimination of $\exists^{\infty}$
(for which see~\cite[Section 2]{cp}) we are lead to ask about the relationship between dp-minimality
and elimination of $\exists^{\infty}$.  For example by~\cite[Lemma 3.3]{gooddp} any
dp-minimal theory expanding that of divisible ordered Abelian groups must eliminate $\exists^{\infty}$.  However in full
generality
this implication is false. Consider the theory $T$ of an equivalence relation with infinitely many infinite classes
together with a finite class of size $n$ for each $n \in \NN$.  It is straightforward to verify that $T$ is dp-minimal
and $T$ obviously does not eliminate $\exists^{\infty}$.  Of course the converse is also false, the random graph eliminates
$\exists^{\infty}$ but is not dp-minimal.

As mentioned earlier Onshuus and Usvyatsov \cite[Corollary 3.8]{alfdpmin} prove that the theory of a Shelah
expansion of an o-minimal structure is dp-minimal.  The
 Shelah expansion
is constructed by beginning with a structure $\mf{M}$ and
an $\card{M}^+$-saturated elementary extension $\mf{N}$ of $\mf{M}$ and expanding $\mf{M}$ to $\mf{M}^*$ by adding predicates
for all sets of the form $X \cap M^m$ where $X \subseteq N^m$ is $\mf{N}$-definable.    It follows
by results in~\cite[Section 1]{shdepfo} that if $\mf{M}$ is o-minimal then $\mf{M}^*$ has weakly o-minimal theory.
Notice of course that the theory of any reduct of $\mf{M}^*$ must also be weakly o-minimal and dp-minimal.

To complete the picture regarding weak o-minimality and dp-minimality
we give an example of a structure $\mf{M}=\langle M, +,<, \dots \rangle$ which is
a model of the theory of divisible ordered Abelian groups, has weakly o-minimal theory, and
so that no model elementary equivalent to $\mf{M}$ may be obtained as a reduct of
a Shelah expansion.   Our example is similar to one exhibited in~\cite[Example 2.6.2]{weakominfield},
differing primarily
 in that we consider a theory expanding that of divisible ordered Abelian groups.

\begin{example} Let $\mf{M}$ be $\langle \QQ \times \QQ, <, +, (0,0), f, \lambda_q\rangle_{q \in \QQ}$ where $<$ is the lexicographic
order, $+$ is componentwise addition, $f$ is a unary function given by $f((a,b))=(-a,b)$ and if  $q \in \QQ$ then $\lambda_q$ is a
unary function given by $\lambda_q((a,b))=(qa,qb)$.
\end{example}

That $\mf{M}$ is a model of the theory of divisible ordered Abelian groups is immediate.

\begin{fact}  $Th(\mf{M})$ admits quantifier elimination.
\end{fact}

\begin{proof}  First notice that any term $t$ is equivalent to a term $t^{\prime}$ so that the only
occurrences of $f$ in $t^{\prime}$ are applied to variables (i.e. if $f(s)$ appears in the construction of
$t^{\prime}$ then $s$ is a variable).  This follows easily since $f$ commutes with addition and multiplication
by a rational.  Next notice that for $a,b \in M$ $a,b$ have the same first coordinate if and only if
$\phi(a,b)$ holds where $\phi(x,y)$ is the formula \[(x=y) \vee (x<y \wedge f(x)<f(y)) \vee (y<x \wedge f(y)<f(x).\]

We reduce to the case that we must eliminate a quantifier from a formula of the form $\exists x (s_0 < x < s_1 \wedge
t_0 < f(x)< t_1)$ where $s_0,s_2, t_0, t_1$ are terms not involving $x$.  (The case where some of the $s_i$ or $t_i$
is $\pm \infty$ is similar.)  It follows that this formula is equivalent to the conjunction of:
\begin{itemize}

\item $s_0<s_1 \wedge t_0<t_1$
\item $\phi(s_0,s_1) \rightarrow (t_0 \leq f(s_0) \wedge f(s_1) \leq t_1)$
\item $(\neg \phi(s_0,s_1) \wedge \neg \phi(t_0,f(s_1)) \wedge \neg \phi(t_1, f(s_0))) \rightarrow f(s_1)<t_0<t_1<f(s_0)$.
\item $(\neg \phi(s_0,s_1) \wedge \phi(t_0,f(s_1))) \rightarrow (t_0 \leq f(s_1) \wedge t_0<t_1)$
\item $(\neg \phi(s_0,s_1) \wedge \phi(t_1,f(s_0))) \rightarrow (f(s_0) \leq t_1 \wedge t_0<t_1)$

\end{itemize} \end{proof}

\begin{cor} $Th(\mf{M})$ is weakly o-minimal.
\end{cor}

\begin{prop}  If $\mf{M}_0$ is an o-minimal reduct of a model of $Th(\mf{M})$  and $\mf{N}$ is
and elementary extension of $\mf{M}_0$ the graph of $f$ is not definable in the structure induced by $\mf{N}$ on $M_0$.
\end{prop}

\begin{proof} Without loss of generality $\mf{N}$ is $\card{M_0}$-saturated.
By~\cite[Section 1]{shdepfo} it suffices to show that for no definable $X \subset (N)^2$ is
$X \cap M_0$ the graph of $f$.  By the o-minimality of $\mf{N}$ we decompose $X$ into cells
$C_1, \dots, C_n$.  By~\cite[Section 1]{shdepfo}) the structure induced on $M_0$ by $\mf{N}$ is weakly
o-minimal and thus for some $1 \leq i \leq n$ and for some $(a,q) \in M_0$
the graph of  $f \restriction ((a,q), \infty)$ is contained in  $C_i \cap M_0$.  There are now two
equally simply cases.  Either $C_i$ is the graph of a continuous (without loss of generality) monotone function or
$C_i$ is the region between the graphs of two continuous functions $g_1$ and $g_2$ which we may also assume to
be monotone.

If $C_i$ is the graph of $g$ then  $g((a+1,q))=(-a-1,q)$ and $g((a+1, q+1))=(-a+1,q+1)$. Thus  $g$ is increasing.
 But also $g((a+2,q))=(-a-2,q)$ and hence $g$ is decreasing, a contradiction.

If $C_i$ is the region between $g_1$ and $g_2$ then $g_1((a+1),q))<(-a-1,q)$ and
$g_1((a+1,q))>(-a-1,p)$ for all $p<q$.  Hence $g((a+1),q))>f((a+1,p))>g((a+1,p))$ and $g$ is increasing.
But $g_1((a+2,q))<(-a-2,q)$ and thus $g$ must be  decreasing, a contradiction.\end{proof}

\section{A complicated dp-minimal divisible ordered Abelian group}\label{nwom}

\maketitle



In this section we focus on theories $T$ which extend that of dense linear ordering and so necessarily contain a symbol
$<$.  A reasonable question arising out of the work found in~\cite{gooddp} is whether every dp-minimal
$T$ expanding the theory of divisible ordered Abelian groups must be weakly o-minimal.  In this section we show via an
example that the
answer to this question is ``no''.

For the example let $\RR((x^{\RR}))$ be the field of generalized power series with real coefficients, real
exponents, and well-ordered supports.  For $a \in \RR((x^{\RR}))$ write
$v(a)$ for its valuation.   We wish to consider only the additive structure of the field augmented with a new relation.
Let $\mc{L}$ be the language consisting of a binary function $+$, a binary relation $<$, a constant $0$, unary functions
$s_q$ for $q \in \QQ$, a unary predicate $P$, and binary predicates $R_n$ for each natural number $n \in \omega$ (including $0$).

Let $\mf{R}$ be the $\mc{L}$ structure with universe $\RR((x^{\RR}))$ where $+,<,0$ are
interpreted in the obvious way.  Interpret the $s_q$ as
multiplication by $q$ for each $q \in \QQ$.  Interpret $P$ as $\{x \in \RR((x^{\RR})) : v(x) \in \ZZ\}$.
Let $R_0$ be the equivalence relation of being in the same connected component of $P$ or of $\RR((x^{\RR})) \setminus P$.  For each $n > 0$, let $R_n$ be the set of all pairs so that either $x \in P$ and there is a sequence $x = x_0 < x_1 < \ldots < x_n = y$
such that $x_i \in P$ if and only if $i$ is even, but
there is no such sequence $x = x_0 < x_1 < \ldots < x_{n+1} = y$; or else  $x \notin P$ and
there is a sequence $x = x_0 < x_1 < \ldots < x_n = y$ such that $x_i \in P$ if and only if $i$ is odd,
but there is no sequence $x = x_0 < x_1 < \ldots < x_{n+1} = y$.
Notice that the $R_n$ are definable in the language with just the group structure and $P$. We add them for quantifier elimination.

Our first goal is to axiomatize $Th(\mf{R})$ and to show this theory has quantifier elimination.
To this end we describe a theory $T$ we intend to show axiomatizes $Th(\mf{R})$.  $T$ consists of:
\begin{enumerate}

\item[(1) ]  The usual axioms for an ordered divisible Abelian group in the language $\left\{+, <, 0\right\}$, and $s_q$ denotes scalar
multiplication by $q$.

\item[(2) ]  $0 \in P$.

\item[(3) ]  $x \in P$ if and only if $-x \in P$.

\item[(4) ]  $P$ and $\neg P$ are open sets.

\item[(5) ]  If $x \leq y$ and the interval $[x,y] \subseteq P$, then for any $\QQ$-linear combination $z$ of $x$ and $y$ with positive
coefficients, $z$ lies in the same connected component of $P$ as $x$ and $y$.

\item[(5')]  If $x \leq y$ and the interval $[x,y] \subseteq \neg P$, then any $\QQ$-linear combination $z$ of $x$ and $y$
with positive
coefficients, $z$ lies in the same connected component of $\neg P$ as $x$ and $y$.

\item[(6)] $R_0$ is a symmetric relation, and if $x \leq y$, then $R_0(x,y)$ holds if and only if the interval $[x,y]$ lies entirely within $P$ or entirely within $\neg P$.

\item[(7) ] For any $x, y$ and positive $n < \omega$, $R_n(x,y)$ holds if and only if $x < y$ and there are exactly $n$ ``alternations
of $P$'' between $x$ and $y$.  More precisely, either:

\begin{enumerate}

\item[(A)] $x \in P$ and there is a sequence $x = x_0 < x_1 < \ldots < x_n = y$ such that $x_i \in P$ if and only if $i$ is even, but
there is no such sequence $x = x_0 < x_1 < \ldots < x_{n+1} = y$; or else

\item[(B)] $x \notin P$ and there is a sequence $x = x_0 < x_1 < \ldots < x_n = y$ such that $x_i \in P$ if and only if $i$ is odd,
but there is no sequence $x = x_0 < x_1 < \ldots < x_{n+1} = y$.

\end{enumerate}

\item[(8) ] For any positive $x$ there is a $y$ such that $R_1(x,y)$.

\item[(8')] For any positive $x$ there is a $y$ such that $R_1(y, x)$.

\end{enumerate}

For  $M \models T$ and  $a \in M$ we introduce some useful notation.  If $a \in P(M)$ (respectively, $\neg P(M)$) and $C$ is the
convex component of $P(M)$ (or $\neg P(M)$) containing $a$, then $[a] = C \cup -C$.
We say $[a] \leq [b]$ if $|a| \leq |b|$,
and let $[a]_{\leq} = \bigcup_{[b] \leq [a]} [b]$.  So the content of Axiom~8 is that the induced ordering on the classes $[a]$
is a discrete ordering with a left endpoint $[0]$ but no right endpoint.

\begin{lem}
\label{classes}
1. If $[a] < [b]$, then $a + b \in [b]$.

2. If $[a] = [b]$, then $a + b \in [a]_{\leq}$.

3. $[a]_{\leq}$ is closed under $\QQ$-linear combinations.
\end{lem}

\begin{proof}
1. If $0 < a < b$, then $a + b > b$, so if $a + b \notin [b]$, then by Axiom~5, $2 b < a + b$.  But this implies that $b < a$,
contradiction.  If $a < 0 < b$, then $0 < -a < b$, so $0 < a+b < b$.  So if $a+b \notin [b]$, then by Axiom~5 again,
$2a + 2b < b$, and $b < 2 (-a)$, a contradiction to Axiom~5.  The other two cases are similar.

2. Without loss of generality $|a| \leq |b|$.  If $0 < a \leq b$, then $0 < a + b \leq 2 b \in [b]$, by Axiom~5.  If $a < 0 < b$,
then $0 < a + b < b$,
and so $a + b \in [b]_{\leq} = [a]_{\leq}$.  The other cases are similar.

3. For any nonzero $q \in \QQ$, $q a \in [a]$ by Axioms~3 and 5.  The rest follows by 2.\end{proof}

\begin{prop}
$T$ is complete and has quantifier elimination.
\end{prop}

\begin{proof}
We prove both statements simultaneously by a back-and-forth argument: suppose that $M$ and $N$ are $\omega$-saturated models
of $T$, $\overline{a} = (a_0, \ldots, a_{n-1}) \subseteq M$, $\overline{b} = (b_0, \ldots, b_{n-1}) \subseteq N$,
and $\tp_{\textup{qf}}(\overline{a}, M) = \tp_{\textup{qf}}(\overline{b}, N)$.
Then for every $a' \in M$,
we show that there is some $b' \in N$ such that
$\tp_{\textup{qf}}((\overline{a}, a'), M) = \tp_{\textup{qf}}((\overline{b}, b'), N)$.

We do this by cases.

Case A: $a'$ is in the $\QQ$-linear span of $\overline{a}$.  Say $a' = \Sigma_{i < n} s_{q_i}(a_i)$.

Let $b' \in N$ to be the corresponding $\QQ$-linear combination of $\overline{b}$.  If we pick $i < n$
such that $q_i \neq 0$ and $[a_i]$ is as large as possible, then by Lemma~\ref{classes}, $[a'] = [a_i]$,
and similarly $[b'] = [b_i]$.  The equality of the quantifier-free types now follows directly.

Case B: $a'$ is not in the $\QQ$-linear span of $\overline{a}$ but there is an element $c$ in the $\QQ$-linear
span of $\overline{a}$ such that $[a'] = [c]$.

Without loss of generality, $a' > 0$.  Let $$S_0 = \left\{x \in \Span_{\QQ}(\overline{a})
\cap [c] : x < a' \right\}$$ and let $$S_1 = \left\{ x \in \Span_{\QQ}(\overline{a}) \cap [c] : a' < x \right\}.$$
For $\ell = 0,1$, let $\langle d^\ell_i : i < \omega \rangle$ be an enumeration of $S_\ell$.
Since $\tp_{\textup{qf}}(\overline{a}, M) = \tp_{\textup{qf}}(\overline{b}, N)$, we can take corresponding sets $T_0$ and $T_1$
in $N$, with corresponding enumerations $\langle e^\ell_i : i < \omega \rangle$.

Let $f^0_i = a' - d^0_i$ and let $f^1_i = d^1_i - a'$ (these are always positive points).  Then, using the
fact that the $[ \cdot ]$-classes in both $M$ and $N$ are discrete linear orderings, we can pick elements $g^\ell_i \in N$
such that:

I. If there is any $x \in \Span_{\QQ}(\overline{a})$ such that $[x] = [f^\ell_i]$ (or $R_n(x, f^\ell_i)$,
or $R_n(f^\ell_i, x)$), then let $y \in \Span_{\QQ}(\overline{b})$ be the corresponding element and pick $g^\ell_i$
such that $[g^\ell_i] = [y]$ (or $R_n(y, g^\ell_i)$, or $R_n(g^\ell_i, y)$);

II. If $x \in \Span_{\QQ}(\overline{a})$ and $[x] < [f^\ell_i]$ (or $[f^\ell_i] < [x]$), then
let $y \in \Span_{\QQ}(\overline{b})$ be the corresponding element, and we require that
$[y] < [g^\ell_i]$ (or $[g^\ell_i] < [y]$).

\begin{claim}
There is an element $b' \in N$ satisfying the conditions:

\begin{enumerate}
\item $b'$ is in the same $[\cdot]$-class as any element of $T_0$ or $T_1$;

\item $T_0 < b' < T_1$;

\item $[b' - e^\ell_i] = [g^\ell_i]$ for any $i < \omega$ and $\ell = 0,1$;

\item $b' \notin \Span_{\QQ}(\overline{b})$.
\end{enumerate}

\end{claim}

(Note that one of $T_0$ or $T_1$ may be empty, so half of the second condition may be vacuous.)

\begin{proof}
First assume $T_0$ and $T_1$ are nonempty, for simplicity.

One case is where $\left\{[g^0_i] : i < \omega \right\}$ and $\left\{[g^1_i] : i < \omega \right\}$
have least elements---then call these elements $[g^0_i]$ and $[g^1_j]$, and without loss of generality $[g^0_i] \leq [g^1_j]$.
If $[g^0_i] < [g^1_j]$ and $[g^0_i]^+$ denotes the positive elements of this class, then pick some $b' \in e^0_i + [g^0_i]^+$ such that:

i. $b' \notin \Span_{\QQ}(\overline{b})$,

ii. $T_0 < b' < T_1$, and

iii. For any $k$ such that $a' - d^0_k \in [f^0_i]$, $b' \in e^0_k + [g^0_i]^+$.

(This is always possible since $[g^0_i]^+$ is infinite, being closed under scaling by positive elements of $\QQ$,
and using compactness and $\omega$-saturation.)  Using Lemma~\ref{classes}, it follows that for any $k < \omega$
and $\ell = 0,1$, $[b' - e^\ell_k] = [g^\ell_k]$.  On the other hand, if $[g^0_i] = [g^1_i]$, then in picking the
element $b'$ as above, we can ensure in addition that $b' \in e^1_j - [g^1_i]^+$, and that for
every $k$ such that $d^1_k - a' \in [f^1_i]$, $b' \in e^1_k - [g^1_i]^+$; this is enough to ensure that $b'$ satisfies
the properties we want.

If $\left\{[g^0_i] : i < \omega \right\}$ and $\left\{[g^1_i] : i < \omega \right\}$ have no least elements, then we can use
compactness and $\omega$-saturation to pick $b' \in N$ such that for every $i < \omega$, $b' \in e^0_i + [g^0_i]^+$
and $b' \in e^1_i - [g^1_i]^+$, and it automatically follows that $b' \notin \Span_{\QQ}(\overline{b})$.  The ``mixed case''
(one of these sets has a least element, the other does not) is handled similarly.

Finally, if $T_0$ is empty, note that $T_1$ cannot have a least element (since if $x \in T_1$ then $\frac{1}{2} x$
must be as well), so the usual compactness argument ensures there is a $b'$ in the right $[\cdot]$-class such that $b'
< T_1$.  As above, we can also ensure that $[b' - e^1_i] = g^1_i$ for each $i < \omega$.  The case where $T_1 = \emptyset$
is symmetric.\end{proof}

With $b'$ as above, Lemma~\ref{classes} ensures that $\tp_{\textup{qf}}(\overline{b}, b') = \tp_{\textup{qf}}(\overline{a}, a')$.

Case C: Cases~A and B fail, but there is some $i < n$ such that $a' \in [a_i]_{\leq}$.

Choose $c_0, c_1 \in \Span_{\QQ}(\overline{a})$ such that $[c_0] < [a'] < [c_1]$  but the ``distances'' are minimized:
that is,
if possible, there is some positive $m$ such that $R_m(c_0, a')$ holds (and similarly for $c_1$), and such a number $m$ is
minimized.  There are subcases: for instance, if $R_{m_0}(c_0, a')$ and $R_{m_1}(a', c_1)$ hold, then we just need to
pick $b' \in N$ such that for the corresponding $d_0, d_1 \in \Span_{\QQ}(\overline{b})$, $R_{m_0}(d_0, b')$ and
$R_{m_1}(b', d_1)$ hold.
On the other hand, if there is no such $m_0$ and no such $m_1$, then compactness and $\omega$-saturation yield a
corresponding $b' \in N$.  The final subcase ($a'$ is a finite distance from one of the $c_i$ but not the other)
is handled in the same way, noting that Axiom~8 ensures that if there are infinitely many components of $P(N)$
between $d_0$ and $d_1$, then there are always elements $b' \in N$ such that $R_k(d_0, b')$ or $R_k(b', d_1)$, for
any $k < \omega$.

Case D: Previous cases fail, but there is some $i < n$ and some positive $m$ such that $R_m(a_i, |a'|)$ holds.

Choose $i$ and $m$ such that $m$ is minimal.  By Lemma~\ref{classes}, the truth values of $R_k(c, a')$ and $R_k(a', c)$
for any $c \in \Span_{\QQ}(\overline{a})$ are now uniquely determined.  By Axiom~8, there is a corresponding $b' \in N$.

Case E: The only remaining case is that $|a'|$ is greater than every $|a_i|$ and $R_m(a_i, |a'|)$ fails for
every possible $i$ and every $m < \omega$.  Then by Lemma~\ref{classes}, the only additional information needed
to determine $\tp_{\textup{qf}}((\overline{a}, a'), M)$ is the sign of $a'$.  A corresponding $b' \in N$
exists by Axiom~8 and $\omega$-saturation.\end{proof}

\begin{cor}\label{open} If $\mf{M} \models T$ and $X \subset M$ is infinite and definable then $X$ has interior yet $T$
is not weakly o-minimal.
\end{cor}

\begin{proof} This follows immediately from the quantifier elimination.\end{proof}

\begin{cor}
$T$ does not have the independence property.
\end{cor}

\begin{proof}
Let $I$ be indiscernible over  $\emptyset$ with uncountable cofinality.  Without loss of generality
 $I$ is increasing, and there
are three possibilities: either all elements of $I$ are in the same $[ \cdot ]$-class, or else we have some
$n < \omega$ such that $R_n(a_i, a_{i+1})$ holds for every $a_i \in I$, or else neither of these hold and there
are infinitely many $[ \cdot ]$-classes between two adjacent elements of $I$.  Suppose that $A$ is any finite set.
Then there are only finitely many $[ \cdot ]$-classes in $\dcl(A)$ (by Lemma~\ref{classes}), and in either of the
three cases, it is straightforward to check using quantifier elimination that some cofinal subsequence of $I$ is indiscernible
over $A$.\end{proof}

\begin{cor}
$T$ is dp-minimal.
\end{cor}

\begin{proof}
Suppose that $t(x, \overline{y})$ is any term in $T$; then we can write $t(x, \overline{y}) = qx + s(\overline{y})$
for some $q \in \QQ$ and some term $s(\overline{y})$.  If $\overline{a}, \overline{b} \in M \models T$, then there is
some $c \in P(M)$ such that $[c] > [s(\overline{a})]$ and $[c] > [s(\overline{b})]$.  By Lemma~\ref{classes},
as long as $q \neq 0$, $M \models P(t(c,\overline{a})) \wedge P(t(c, \overline{b}))$.  So for any $t(x, \overline{y})$
that depends nontrivially on $x$,
$$T \vdash \forall \overline{y}_0 \forall \overline{y}_1 \exists x
\left[P(t(x, \overline{y}_0)) \wedge P(t(x, \overline{y}_1)) \right].$$  The same argument works with $\neg P$ or $\neg R_n$
in place of $P$, and for finite conjunctions of such formulae.

This means that if $\left\{\varphi(x; \overline{a}_i) : i < \omega \right\}$ is a $k$-inconsistent sequence
of formulae in $T$, each of which is a conjunction of atomic formulae and negated atomic formulae, then
each $\varphi(x; \overline{a}_i)$ can be assumed to be a conjunction of formulae of the following
three forms:
 $t_0(x; \overline{a}_i) < t_1(x; \overline{a}_i)$, $\neg \left(t_0(x; \overline{a}_i) < t_1(x; \overline{a}_i) \right)$,
 or $R_n(t_0(x; \overline{a}_i), t_1(x; \overline{a}_i))$.  So each $\varphi(x; \overline{a}_i)$
is a finite union of convex sets.  From here we can argue as in the  weakly o-minimal case to show that $T$
is inp-minimal.\end{proof}

Thus we have a non-weakly o-minimal theory $T$ extending the theory of divisible ordered Abelian groups which is dp-minimal.
Notice though that Corollary~\ref{open} indicates that infinite definable subsets in models of $T$ are  not too complicated, namely they
must have interior.  Pierre Simon~\cite{simon} has recently shown that this must be the case, specifically an infinite
definable subset of a dp-minimal divisible ordered Abelian group must have interior.

\section{dp-minimality of the p-adics}

In this section our main goal is to show that
if $\QQ_p$ is a p-adic field then $Th(\QQ_p)$ is dp-minimal.
The main technical portion of our proof is written in a more general context than that
of the p-adics in the hopes that our proof of the dp-minimality of the p-adics may be generalized
 to other Henselian valued fields.  We begin by establishing some notation and fixing the context
 for our work.

In what follows, $\NN$ denotes the positive integers.

Let $K$ be a Henselian valued field, considered as a one-sorted
structure in the language $\mathcal L_\text{vf} = \mathcal
L_\text{ring} \cup \{ v(x) \leq v(y) \}$. Write $v: K^\times \to
\Gamma$ for the valuation and value group. We assume that $K$ is
elementarily equivalent to a valued field with value group $\ZZ$; for
example, $K$ is an elementary extension of a $p$-adic field, which is
the application below.

Let $R$ be the valuation ring of $K$ and
let $\mm$ be its maximal ideal.

Fix $k\in \NN$. Note that $1+\mm^k$ is a definable subgroup of
$K^\times$. Adapting notation from Hrushosvki,
we write $RV_k(K)$ for the quotient $K^\times / (1+\mm^k)$.
Let $\pi_k: K^\times\to RV_k(K)$ be the quotient map.
For notational convenience, we
define $\pi_k(0)$ to be a new element $\infty$ which we adjoin to
$RV_k(K)$.


An easy calculation establishes the following proposition.
\begin{prop}
  For all $z\in K$ and all $x,y\in K\setminus\{z\}$,
  $\pi_k(x-z) = \pi_k(y-z)$ if and only if
  $v(x-y)\geq v(y-z)+k$. \label{prop-pi}
\end{prop}

We note that to establish this simple proposition we need to use the fact that our theory has
a model whose value group if $\ZZ$.

We know isolate a special kind of formula whose definition is intended to mimic the form
of formulas defining cells (as in \cite{padiccell}) in the p-adic fields.

\begin{definition}
  The formula $\phi(x;y_0,\dots,y_l)$ is \emph{cell-like} if there is $k\in
  \NN$ such that $\phi(x;y_0\vec{y})\andd \pi_k(x-y_0) = \pi_k(x'-y'_0)$
  implies $\phi(x';y_0'\vec{y})$. We call $y_0$ the \emph{center} of
  $\phi(x;y_0,\dots,y_n)$.
\end{definition}

Let $K$ be as above, and assume further that the residue field
$\overline{K}$ is finite and $K$ is $(2^\omega)^+$-saturated.
We show that there is no ICT
pattern in $K$ for which each formula is cell-like.
On the contrary, assume that in $K$ we
have an ICT pattern as follows:
\begin{gather*}
C_{1,1}, C_{1,2}, C_{1,3},\dots\\
C_{2,1}, C_{2,2}, C_{2,3},\dots
\end{gather*}
where $C_{i,j} = \{ x\in K : \phi_i(x,d_{i,j})\}$ and
$\phi_i(x;\vec{y})$ is cell-like. Write $c_{i,j}$ for the center
of $C_{i,j}$ (i.e.\ $c_{i,j}=(d_{i,j})_0$).  We may assume that $\langle
d_{1,j} \rangle_j$ and $\langle d_{2,j} \rangle_j$ are indexed by the
reals and are
mutually indiscernible sequences. Choose $k$ large enough so that
it witnesses that both $\phi_1$ and $\phi_2$ are cell-like.
We write $\pi$ as an abbreviation for $\pi_k$.

\begin{prop}
   $v(c_{1,1}-c_{2,1})=v(c_{1,2}-c_{2,1})$ and  $v(c_{2,1}-c_{1,1})=v(c_{2,2}-c_{1,1})$. \label{prop:branch}
 \end{prop}

 \begin{proof}
   By symmetry, it suffices to establish the first equality.
   Assume that $v(c_{1,1}-c_{2,1})\neq v(c_{1,2}-c_{2,1})$; without
   loss of generality,
   $v(c_{1,1}-c_{2,1})<v(c_{1,2}-c_{2,1})$. By indiscernibility,
   $\langle v(c_{1,j}-c_{2,1}) \rangle_j$ is an increasing sequence in
   $\Gamma$.
   Because $\langle c_{1,j} \rangle_j$ is uncountable and
   indiscernible over $c_{2,1}$, we see that each
   $v(c_{1,j}-c_{2,1})$ must be in a
   distinct Archimedean class. In particular, for $j\in \ZZ$ we have
   \begin{equation}
     \cdots \ll v(c_{1,-1}-c_{2,1}) \ll v(c_{1,0}-c_{2,1}) \ll v(c_{1,1}-c_{2,1})\ll \cdots. \label{eq:6}
   \end{equation}

   Note that by the ultrametric inequality, we know
   $v(c_{1,j}-c_{2,1})=v(c_{1,j}-c_{1,j+1})$, for all $j\in \ZZ$.
   Since, $\langle c_{2,j} \rangle_j$ is indiscernible over
   $\{c_{1,j}\}_j$, we have that
   $v(c_{1,j}-c_{2,2})=v(c_{1,j}-c_{1,j+1})=v(c_{1,j}-c_{2,1})$. Thus,
   for each $j\in \ZZ$, $v(c_{2,2}-c_{2,1}) \geq v(c_{1,j}-c_{2,1})$.
   (Otherwise, by the ultrametric inequality, $v(c_{1,j}-c_{2,2})=
   v(c_{2,2}-c_{2,1})$; hence, from the indiscernibility of $\langle c_{1,j}
   \rangle_j$ over $\{ c_{2,j} \}_j$, we have $v(c_{1,1}-c_{2,2})=
   v(c_{1,2}-c_{2,2})$, which is impossible.) A fortiori, for each
   $j\in \ZZ$, $v(c_{2,2}-c_{2,1}) \gg v(c_{1,j}-c_{2,1})$.

   Choose $a\in C_{2,1} \cap (K\setminus C_{2,2}) \cap C_{1,1} \cap
   (K\setminus C_{1,2})$.
   Note that $\pi(a-c_{2,1})\neq \pi(a-c_{2,2})$, i.e.   $v(c_{2,2}-c_{2,1})<
   v(a-c_{2,1})+k$. Hence, $v(a-c_{2,1})\gg v(c_{1,j}-c_{2,1})$ for
   all $j\in \ZZ$. However, this implies that
   $\pi(a-c_{1,j})=\pi(c_{2,1}-c_{1,j})$. In particular, from the
   choice of $a$, we have $c_{2,1}\in C_{1,1}$ and $c_{2,1}\notin
   C_{1,2}$. However, this contradicts the indiscernibility of
   $\langle d_{1,j}\rangle_j$ over $c_{2,1}$.
 \end{proof}

\begin{prop}
$v(c_{1,1}-c_{1,2})\gg v(c_{1,1}-c_{2,1})$ and $v(c_{1,1}-c_{1,2})\gg v(c_{1,1}-c_{2,1})$.\label{prop:y}
\end{prop}

\begin{proof}
Again, by symmetry, it suffices to prove the first inequality.
Via indiscernibility, Proposition~\ref{prop:branch} yields
 $v(c_{1,j}-c_{2,1})=v(c_{1,1}-c_{2,1})$ for all $j\in \NN$. Since the
 residue field is finite, a pigeon-hole argument shows that for each
 $k'\in \NN$ there are distinct
 $j,j' \in \NN$ such that
 $\pi_{k'}(c_{1,j}-c_{2,1})=\pi_{k'}(c_{1,j'}-c_{2,1})$. In fact, by
 indiscernibility, we have $\pi_{k'}(c_{1,1}-c_{2,1})=\pi_{k'}(c_{1,2} -
 c_{2,1})$ for all $k'$. In other words, $v(c_{1,1}-c_{1,2})\gg
 v(c_{1,1}-c_{2,1})$, as desired.
\end{proof}

 Choose $a\in C_{1,1} \cap (K\setminus C_{1,2}) \cap C_{2,1} \cap
 (K\setminus C_{2,2})$. Then, since $\pi(a-c_{1,1})\neq
 \pi(a-c_{1,2})$, we know  $v(c_{1,2}-c_{1,1})<
 v(a-c_{1,1})+k$. By Propositions~\ref{prop:branch} and \ref{prop:y},
 $v(a-c_{1,1})\gg  v(c_{1,1}-c_{2,1})=v(c_{1,1}-c_{2,2})$. Hence,
 $\pi(a-c_{2,1})=\pi(c_{1,1}-c_{2,1})$ and
 $\pi(a-c_{2,2})=\pi(c_{1,1}-c_{2,2})$. By the choice of $a$, we see
 $c_{1,1}\in C_{2,1}$ and $c_{1,1}\notin C_{2,2}$, contradicting the
 indiscernibility of $\langle d_{2,j} \rangle_j$ over $c_{1,1}$.
We have thus established:

\begin{thm}\label{celllike} If $K$ is a Henselian valued field elementary equivalent to a valued field
with value group $\ZZ$ and $K$ has finite residue field then $K$ has no ICT
pattern consisting of cell-like formulae.
\end{thm}

We are now ready to establish the dp-minimality of $\QQ_p$.  We recall some essential
facts on the p-adic fields necessary for our proofs.

Fix a prime $p$, and let $\QQ_p$ be the $p$-adic field, considered as
a one-sorted structure in the language of rings. Write $v:
\QQ_p^\times \to \ZZ$ for the $p$-adic valuation. Recall that the
relation $v(x)\leq v(y)$ is definable in the language of rings
\cite[Section 2]{padiccell}. Let $K$ be a sufficiently saturated elementary
extension of $\QQ_p$; all of the hypotheses of Theorem~\ref{celllike} hold of $K$. Thus, there is no cell-like ICT pattern in
$K$. To show $\QQ_p$ is dp-minimal, it suffice to show that if there is an ICT
pattern in $K$, then there is a cell-like ICT pattern in $K$.



\begin{definition}\mbox{}

  \begin{enumerate}
  \item An {\em annulus} in $K$ is a set of the form
    \begin{equation*}
      \Ann(c,\gamma,\delta)=\{ x\in K: \gamma \geq v(x- c) \geq \delta \},
    \end{equation*}
    where $\gamma\in \Gamma\cup \{ \infty\}$, $\delta\in \Gamma \cup \{
    -\infty \}$ and $c \in K$.
  \item Let $P_n$ be set of $n^\text{th}$ powers in $K$.
    A {\em power coset} in $K$ is a set of the form
    \begin{equation*}
      \Pow_{n,\lambda}(c)= \{ x\in K : x-c \in \lambda P_n\},
    \end{equation*}
    where, $n\in \NN$, $\lambda\in \NN \cup \{0\}$, and $c\in K$.
  \item  A {\em
      cell} in $K$ is a non-empty set of the form
    \begin{equation*}
      \Cell_{n,\lambda}(c,\gamma,\delta)= \Ann(c,\gamma,\delta) \cap \Pow_{n,\lambda}(c).
    \end{equation*}
We call $c$ the {\em center} of $\Ann(c,\gamma,\delta)$.
  \end{enumerate}
\end{definition}

\begin{remark}
For each $n$ there are only finitely many cosets of $P_n$ in
$K$, each represented by some $\lambda\in \NN$. Thus, we
consider $\lambda$ to be a term in the language, rather than a
parameter.
\end{remark}


\begin{prop}
A cell is cell-like.
\end{prop}

\begin{proof}
  It is clear that $\Ann(c,\gamma, \delta)$ is cell-like
  with center $c$ (as witnessed by any $k\in \NN$). It suffices to show
  that $\Pow_{n,\lambda}(c)$ is also cell-like with center $c$.
  Using the Hensel-Rychlik Theorem (see \cite{valbook}) one can show that
  for each $n\in \NN$ there is $k\in \NN$ such that for all
  $\lambda\in \NN$, if $z\in \lambda P_n$ and $\pi_k(z)=\pi_k(z')$, then
  $z' \in\lambda P_n$. Thus, this $k$ witnesses that
  $\Pow_{n,\lambda}(c)$ is cell-like with center $c$.
\end{proof}

By the cell-decomposition theorem (see \cite{padicanaly}), every
definable subset of $K$ is a finite union of cells. Suppose there was
an ICT pattern in $K$. Then, by Fact~\ref{union}, we
get an ICT pattern of cells in $K$. However, this contradicts the
results of the previous section and hence we have:

\begin{thm} $Th(\QQ_p)$ is dp-minimal.
\end{thm}

\bibliography{../alfref}
\bibliographystyle{plain}

\end{document}